\documentclass[a4paper,10pt,intlimits,ngerman]{amsart}
\usepackage[varg]{txfonts}
\usepackage[unicode]{hyperref}
\usepackage{graphicx}
\usepackage{color}
\usepackage{tikz-cd}
\usetikzlibrary{calc,patterns,angles,quotes}


\newtheorem{theorem}{Theorem}[section]

\newtheorem{lemma}[theorem]{Lemma}

\hypersetup{
breaklinks=true,
colorlinks=true,
linkcolor=blue,
citecolor=blue,
urlcolor=blue,
}

\numberwithin{equation}{section}

\DeclareMathOperator{\spt}{spt}

\newcommand{\be}{\begin{equation}}
\newcommand{\ee}{\end{equation}}

\newcommand{\vecc}{\mbox{$ c $}}
\newcommand{\vece}{\mbox{ $ e $}}
\newcommand{\vecf}{\mbox{ $ f $}}
\newcommand{\veco}{\mbox{ $ o $}}
\newcommand{\vect}{\mbox{ $ t $}}
\newcommand{\vecu}{\mbox{ $ u $}}
\newcommand{\vecv}{\mbox{ $ v $}}
\newcommand{\vecw}{\mbox{ $ w $}}
\newcommand{\vecx}{\mbox{ $ x $}}

\renewcommand{\epsilon}{\varepsilon}

\newcommand{\essinf}{\mathrm{ess} \, \inf}

\title[Discretization of the M\"obius energy]{A M\"obius invariant discretization of O'Hara's M\"obius energy}


\author{
Simon Blatt}
\address[Simon Blatt]{Departement of Mathematics, Paris Lodron Universit\"at Salzburg, Hellbrunner Strasse 34, 5020 Salzburg, Austria}
\email[Simon Blatt]{simon.blatt@sbg.ac.at}

\author{Aya Ishizeki}
\address[Aya Ishizeki]{Department of Mathematics and Informatics, Faculty of Science, Chiba University
1-33 Yayoi-cho, Inage, Chiba, 263-8522, Japan}
\email[Aya Ishizeki]{ a.ishizeki@chiba-u.jp}

\author{Takeyuki Nagasawa}
\address[Takeyuki Nagasawa]{Departement of Mathematics, Graduate School of Science and Engineering, Saitama University, Saitama 338-8570, Japan}
\email[Takeyuki Nagasawa]{tnagasaw@rimath.saitama-u.ac.jp}

\date{\today}

\begin{document}


\begin{abstract}
We introduce a new discretization of O'Hara's M\"obius energy. In contrast to the known discretizations of Simon and Kim and Kusner it is invariant under M\"obius transformations of the surrounding space. The starting point for this new discretization is the cosine formula of Doyle and Schramm.
We then show $\Gamma$-convergence of our discretized energies to the M\"obius energy under very natural assumptions. 
\end{abstract}

\maketitle

\tableofcontents

\section{Introduction}

Motivated by the work of Fukuhara \cite{Fukuhara1988}, Jun O'Hara in a series of papers \cite{OHara1991,OHara1992,OHara1994} studied the energy 
\begin{align*}
 E(f) := \int_{\mathbb R / \mathbb Z} \int_{\mathbb R / \mathbb Z} \left( \frac 1{\|f(x) - f(y)\|^2} - \frac 1 {d_f(x,y)^2} \right) \|f'(x)\| \cdot \|f'(y)\| dx dy
\end{align*}
of a regular closed curve $f \in C^{0,1}(\mathbb R / \mathbb Z, \mathbb R^n)$.
In the formula above $d_f(x,y)$ denotes the length of the shortest arc of the curve $f$ connecting the two points $f(x)$ and $f(y)$.
 This has been the first and still is the most prominent and best known geometric knot energy. Due to its invariance under M\"obius transformations found by Freedman, He, and Wang \cite{Freedman1994}, it is now called \emph{M\"obius energy}.  
 
There have been two successful attempts to discretize this energy. Jonathan Simon \cite{Simon1994} defined the so-called \emph{minimal distance} energy for polygons $p$ by
\begin{equation*}
 U_{MD}(p) := \sum_{X,Y: \text{non-consecutive segments of } p} \frac{l(X)\cdot l(Y)}{MD(X,Y)^2}
\end{equation*}
where $l$ denotes the length of the segment and $MD(X,Y)$ is the minimal distance of points on $X$ and $Y$. Among many other things, Simon proved that this energy can be minimized within knot classes and showed together with Eric Rawdon some explicit error bounds of $C^{1,1}$- curves.

A simpler discretization was suggested by Kim and Kusner in \cite{Kim1993}, namely
\begin{equation*}
 E^n := \sum_{i,j=1}^n \left( \frac 1 {\|p(\theta_i) - (\theta_j)\|^2} - \frac 1 {d(\theta_i, \theta_j)^2}\right) d(\theta_{i+1}, \theta_i) d(\theta_{j+1}, \theta_j).
\end{equation*}
Here, $p(\theta_i)$ denote the vertices of the polygon $p$ and $d$ the distance of points along the polygon $p$.
Sebastian Scholtes could prove that this discretization $\Gamma$-converges to the M\"obius energy in \cite{Scholtes2014}. He furthermore showed that this energy is minimized by regular $n$-gons.

While this discretization does not immediately have self-repulsive effects, it has the big advantage that one can actually show that it is minimized by regular $n$-gons -- a question that is still open for Jonathan Simon's minimal distance energy.

In contrast to the M\"obius energy itself none of the discrete versions known up to now is invariant under M\"obius tranformations. Of course the absence of this invariance  has  some positive effects on the behavior of these energies. For example it makes  the proof that minimizers exist much easier than for the M\"obius energy itself. But on the other hand these discrete versions will not reflect some of the fundamental properties of the continuous M\"obius energy. Motivated by this observation, we will introduce and study a M\"obius invariant discrete version of the M\"obius energy, prove some basic properties like $\Gamma$-convergence to the M\"obius energy, and formulate some open problems.

Let us present the outline of this article: In Section \ref{sec:DiscreteEnergy} we recapitulate the cosine formula for the M\"obius energy due to Doyle and Schramm, discuss its M\"obius invariance, and describe how to use it to discretize the M\"obius energy in a M\"obius invariant way. Futhermore, we will see that this energy is minimized by polygons with vertices on a circle. Section \ref{sec:Approx} gathers some facts we will need later on to extend the $\limsup$ inequality to curves with merely bounded M\"obius energy. We recall and slightly extend for this purpose some approximation results from \cite{Blatt2016b}.
Section \ref{sec:Gamma} contains the statement and proof of our main result, the $\Gamma$-convergence for the discretized energies to the M\"obius energy. In the appendix we derive a formula for the discretized energies for future reference.

\section{A Discrete M\"obius Invariant Version of the M\"obius Energy} \label{sec:DiscreteEnergy}

\subsection{Cosine Formula of Doyle and Schramm}

As communicated to us by Kusner and Sullivan \cite{Kusner1997}, Doyle and Schramm found a new interpretation of O'Hara's M\"obius energy that highlights its M\"obius invariance.

They put together two M\"obius invariant pieces to build the M\"obius energy. The first piece is the quantity
\begin{equation*}
 \frac {\|f'(x)\| \cdot \|f'(y)\|}{\|f(x)-f(y)\|^2}
\end{equation*}
which is obviously invariant under scalings and translations. To derive the invariance under the inversion on the unit sphere 
$$
 I : \mathbb R^n \setminus \{0\} \rightarrow \mathbb R^n \setminus \{0\}, \quad x \mapsto \frac x {\|x\|^2}
$$
we calculate
\begin{equation} \label{eq:moebInv}
\begin{aligned}
 \|(I \circ f) (x) - (I \circ f(y)\|^2 & = \frac {\|f(y)\|^4\|f(x)\|^2 -2 \|f(x)\|^2 \|f(y)\|^2 \langle f(x), f(y) \rangle + \|f(y)\|^2 \|f(x)\|^4} {\|f(x)\|^4 \|f(y)\|^4}
 \\ 
 & = \frac{\|f(x) - f(y)\|^2}{\|f(x)\|^2 \|f(y)\|^2} 
\end{aligned}
\end{equation}
and
\begin{align*}
 \|(I \circ f)'(x)\| =  \frac{ \|f'(x)\|}{\|f(x)\|^2}.
\end{align*}
We see that 
\begin{equation*}
\frac {\|(I \circ f)'(x)\| \cdot \|(I\circ f)'(y)\|}{\|(I \circ f)(x)-(I \circ f)(y)\|^2} = \frac {\|f'(x)\| \cdot \|f'(y)\|}{\|f(x)-f(y)\|^2},
\end{equation*}
i.e. the term $\|f'(x)\| \cdot \|f'(y)\|/\|f(x)-f(y)\|^2$ is also invariant under inversions on spheres and hence under all M\"obius transformations of $\mathbb R^n \cup \{\infty\}.$

Let us describe the second building block of the cosine formula. Given two points $x, y \in \mathbb R / \mathbb Z$ one considers the circle $C_f(x,y)$ going through the points $f(x)$ and $f(y)$ and being tangent to $f$ at $f(x).$ Then the angle $\alpha_f(x,y)$ at which the two circles $C_f(x,y)$ and $C_f(y,x)$ meet is invariant under M\"obius transformations, as the circles are invariant under such transformations. 

By the observations above, the quantity
\begin{align*}
 E_{\cos}(f) := \int_{\mathbb R / \mathbb Z} \int_{\mathbb R / \mathbb Z} \frac {1-\cos(\alpha_f(x,y))}{\|f(x)-f(y)\|^2} \|f'(x)\| \|f'(y)\| dx dy
\end{align*}
is invariant under M\"obius transformations. Strikingly, we have (cf. \cite{Kusner1997})
\begin{align*}
 E(f)=E_{\cos} (f) +4.
\end{align*}

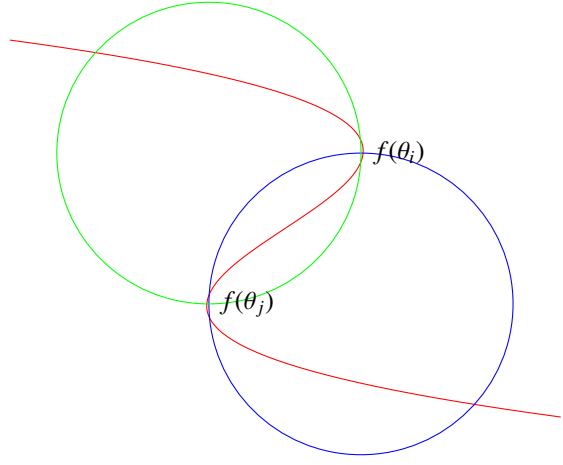
\begin{figure}
\begin{tikzpicture}[scale=1,rotate=90]
\draw [red, domain=-2.5:2.5, samples=100] plot (\x, {2*(sqrt(5)-2) * \x * \x * \x - 1.5 *  \x });
\draw [blue] (-1,-1) circle (2);
\draw [green] (1,1) circle (2);
\node at (1,-1.5) {$f(\theta_i)$};
\node at (-1,0.5) {$f(\theta_j)$};
\end{tikzpicture}
\caption{This picture illustrates the definition of the angle $\alpha_{ij}$. }
\end{figure}

\subsection{Discretizing the Cosine Formula}

Let us discretize both ingredients to the cosine formula in a  M\"obius invariant way. Let us consider a closed polygon with vertices $p(\theta_i),$ $i=1, \ldots m,$ and let us introduce the shortcuts
$$
 \Delta _i^j p:= p(\theta_j) - p(\theta_i) , \quad \Delta_i p = \Delta_i^{i+1}p= p(\theta_{i+1}) - p(\theta_{i}).
$$
Then a M\"obius invariant version of the first ingredient for $i\not=j$ is the so-called \emph{cross ratio}
\begin{align*}
 \frac {\|\Delta_i p\| \cdot \| \Delta_j p\|}{\|\Delta_i^j p\| \cdot \|\Delta_{i+1}^{j+1} p\|}
\end{align*}
which can easily be checked using equation \eqref{eq:moebInv}.

Let us now assume that the four points $p(\theta_i),p(\theta_j),p(\theta_{i+1}),$ and $p(\theta_{j+1})$ are pairwise different. To generalize the second ingredient to the cosine formula, we consider the circles 
$$
 C_{i,j}
$$
going through the points $p(\theta_i), p(\theta_{i+1})$ and $p(\theta_j)$. Let $\alpha_{ij}$ be the angle at which the circles $C_{i,j}$ and $C_{j,i}$ meet and $\tilde \alpha_{ij}$ be the angle at which the circles $C_{i,(j+1)}$ and $C_{j,(i+1)}$ meet. By our construction, these angles are invariant under M\"obius transformations. 

We set
$$
 E^{m}_{cos}(p) = \sum_{d_m(i,j)>1} \frac {\|\Delta_i p\|\cdot \|\Delta_j p\|}{\|\Delta_i^j p\|\|\Delta _{i+1}^{j+1} p\|} \left( 1 - \frac 1 2 \left(\cos (\alpha_{ij}) + \cos (\tilde \alpha_{ij}) \right) \right).
$$

One observes that the summands in the definition of our discretization $E^{m}_{cos}(p)$ are all nonnegative and that the $E^m_{cos}(p)$ is zero if and only if the angle between all the circles $C_{ij}$ and $C_{ji}$ is zero and hence if all the vertices of the polygon $p$ lie on one circle.

Note that there are obvious different choices for the integrals for which the technics developed within this article work as well.

\section{Approximation by Smooth Curves} \label{sec:Approx}

Our proof of the $\limsup$-inequality later on will rely in an essential way on the following extension of an approximation result in \cite{Blatt2016b}. Let $\eta \in C^\infty(\mathbb R, [0,\infty[)$ be such that $\spt \eta \subset [-1,1]$ and $\int_{\mathbb R} \eta dx =1$. Furthermore, let $\eta_\varepsilon (x) = \frac 1 \varepsilon \eta(\frac x \varepsilon)$ for all $x \in \mathbb R$, $\varepsilon >0$.
For an embedded curve $f \in W^{3/2 ,2 } (\mathbb R / \mathbb Z, \mathbb R^n) \cap W^{1,\infty}(\mathbb R / \mathbb Z, \mathbb R^n)$ that is regular in the sense that $\mathrm{ess}\inf |f'| >0$ and $\varepsilon >0$ we consider the standard smoothend curves
$$
 f_\varepsilon := f \ast \eta_\varepsilon.
$$
The following theorem holds.

\begin{theorem} \label{thm:Approximation}
 The mappings $f_\varepsilon$ are regular curves for $\varepsilon >0$ small enough converging to $f$ in $W^{3/2,2}$ for $\varepsilon \rightarrow 0$ and
 $$
  E(f_{\varepsilon}) \xrightarrow{\varepsilon \downarrow 0} E(f).
 $$
\end{theorem}

For the convenience of the reader and since it is of great importance in the following, we give a complete proof of this result.

We start by proving the following uniform bi-Lipschitz bound.

\begin{lemma} \label{lem:bilipschitz}
 In the situation above, there is a constant $\varepsilon_0 >0$ and a constant $C < \infty$ such that
 $$
  \sup_{x,y \in \mathbb R, |x-y|\leq \frac 1 2 , 0 \leq \varepsilon <\varepsilon_0} 
  \frac{\|f_\varepsilon(x) - f_\varepsilon(y)\|}{ |x-y| } >0
 $$
 and
 $$
  \sup_{x,y \in \mathbb R / \mathbb Z, 0 \leq \varepsilon <\varepsilon_0} 
  \frac{\|f_\varepsilon(x) - f_\varepsilon(y)\|}{ d_{f_\varepsilon}(x-y)| } >0.
 $$
\end{lemma}

Of course Lemma~\ref{lem:bilipschitz} implies that the curves $f_\varepsilon$ are embedded regular curves for $\varepsilon \in [0, \varepsilon_0)$ with
$$
 \essinf \| f_\varepsilon'  \|  \geq C^{-1}.
$$

\begin{proof}
Let $m := \essinf \| f'  \| >0$, $M := \sup_{x \in \mathbb R} \eta(x) + 1 >1$, and 
$$
 D_f(r):=\sup_{z \in \mathbb R / \mathbb Z} \left( \int_{B_{r}(z)} \int_{B_r(z)} \frac {\|f'(x) - f'(y)\|^2} {|x-y|^2} dx dy \right)^{\frac 1 2}.
$$
Note that $D_f(r) \rightarrow 0 $ as $r \rightarrow 0.$
We first observe that $\|f'_{\varepsilon}\|_{L^\infty} \leq \|f'\|_{L^\infty}$ and using $f_\varepsilon '= f' \ast \eta_\varepsilon$,
Jensen's, and Fubini's theorem
\begin{align*}
 \int_{B_{r}(z)} \int_{B_r(z)} \frac {\|f_{\varepsilon}'(x) - f_{\varepsilon}'(y)\|^2} {|x-y|^2} dx dy & =
  \int_{B_{r}(z)} \int_{B_r(z)} \frac {\|\int_{\mathbb R}(f'(x-\xi) - f'(y-\xi)) \eta_\varepsilon(\xi) d\xi \|^2} {|x-y|^2} dx dy \\
  &\leq  \int_{B_{r}(z)} \int_{B_r(z)} \int_{B_\varepsilon(0)} \frac {\|f'(x-\xi) - f'(y-\xi)\|^2 } {|x-y|^2} \eta_\varepsilon(\xi) d\xi  dx dy \\
  &=  \int_{B_\varepsilon(0)} \int_{B_{r}(z-\xi)} \int_{B_r(z-\xi)} \frac {\|f'(x) - f'(y)\|^2 } {|x-y|^2} \eta_\varepsilon(\xi)   dx dy d\xi 
\end{align*}
and hence
\begin{equation} \label{eq:DBounded}
 D_{f_{\varepsilon}} (r) \leq D_{f}(r).
\end{equation}
We get for $0< \varepsilon \leq \frac 1 2 $ using again Jensen's inequality
\begin{align*}
 \frac 1 {2\varepsilon} \int_{B_\varepsilon(z)} \|f'(x) - f_\varepsilon'(z) \| dx
 &\leq
 \frac 1 {2\varepsilon} \int_{B_\varepsilon(z)} \int_{B_\varepsilon(0)}  \|f'(x) - f'(y)\| \eta_\varepsilon(z-y) dx dy 
 \\& \leq  
 \left( \frac 1 {(2\varepsilon)} \int_{B_\varepsilon(z)} \int_{B_\varepsilon(z)}  \|f'(x) - f'(y)\|^2 \eta_\varepsilon(z-y) dx dy \right)^{\frac 1 2}
 \\ & \leq \sqrt{2M}
  \left(  \int_{B_\varepsilon(z)} \int_{B_\varepsilon(z)}  \frac{\|f'(x) - f'(y)\|^2}{|x-y|^2} dx dy \right)^{\frac 1 2} \\ & \leq \sqrt{2M} D_f(\varepsilon).
\end{align*}
This implies
 \begin{align*}
 \|f'_{\varepsilon} (z)\| \geq  \frac 1 { 2 \varepsilon }  \int_{B_\varepsilon(z)} \|f'(x)\| dx
 -\frac 1 { 2 \varepsilon }  \int_{B_\varepsilon(z)}
	 \|f'(x) - f'_\varepsilon(z)\| dx
 \geq m - \sqrt {2M} D_f(\varepsilon)
\end{align*}
which is positive, if $\varepsilon>0$ is small enough, as $D_f(\varepsilon)$ converges to $0$ as $\varepsilon$ goes to $0$. Hence, $f_\varepsilon$
 is a regular curve for all $\varepsilon >0$ small enough 

We can now prove the bi-Lipschitz estimate by a similar argument. 
We know that 
\begin{equation*}
 D_{f_\varepsilon} (r) \rightarrow 0
\end{equation*}
as $r \rightarrow 0$ uniformly in $\varepsilon.$
We have for all $0 < r, \varepsilon \leq \frac 1 4$ 
\begin{align*}
  \frac 1 {2r} \int_{B_r(z)} \|f'(x) - \overline {(f_\varepsilon')}_{B_{r}(z)} \| dx
 &\leq
 \frac 1 {(2r)^2} \int_{B_r(z)} \int_{B_r(z)}  \|f_\varepsilon'(x) - f_\varepsilon'(y)\|  dx dy 
 \\& \leq  
 \left( \frac 1 {(2r)^2} \int_{B_r(z)} \int_{B_r(z)}  \|f_\varepsilon'(x) - f_\varepsilon'(y)\|^2  dx dy \right)^{\frac 1 2}
 \\ & \leq 
  \left(  \int_{B_r(z)} \int_{B_r(z)}  \frac{\|f_\varepsilon'(x) - f_\varepsilon'(y)\|^2}{|x-y|^2} dx dy \right)^{\frac 1 2} \\ & \leq D_{f_{\varepsilon}} (r).
\end{align*}
Hence, we get for all $x,y \in \mathbb R$ with $|x-y| \leq 2\varepsilon_0$ setting $z = \frac {x+y} 2$, $r =\frac {|x-y|} 2$
\begin{align*}
 \frac {\|f_\varepsilon(x) - f_\varepsilon(y)\|} {|x-y|} &=  \frac 1 {2r}\left\|\int_{B_r(z)} f'_\varepsilon (\tau) d\tau \right\|
 =  \left\| \overline{(f_\varepsilon')}_{B_r(z)} \right\| 
 \\ & = \frac 1 {2r} \int_{B_r(z)} \| f_\varepsilon'(\tau) \|d\tau - \frac 1 {2r} \int_{B_r(z)} \left| \|f_\varepsilon'(\tau) \| - \| \overline{(f_\varepsilon')}_{B_r(z)} \| \right| d\tau 
 \\ &\geq \frac 1 {2r} \int_{B_r(z)} \| f_\varepsilon'(\tau) \| d\tau - \frac 1 {2r} \int_{B_r(z)} \|f_\varepsilon'(\tau) -  \overline{(f_\varepsilon')}_{B_r(z)}\| d\tau  
 \\ &\geq m - \sqrt{2M} D_{f}(r) -  D_{f_{\varepsilon}}(r)
 \\ &\geq m - (\sqrt{2M}  +1) D_{f}(r) ,
\end{align*}
where we have used \eqref{eq:DBounded} in the last step. Again, if $r>0$ is small enough the righthand side is positive.

As the continuity together with the injectivity of $f $ implies that $\sup_{ \varepsilon_0 \geq |x-y| \leq \frac 1 2} \frac {\|f(x) - f(y)\|}{|x-y|} >0$
the first inequality is proven.

The second inequality follows from the first one using the bound $0 <\frac m4 \leq \|f'_\varepsilon\| \leq \|f'\|_{L^\infty} < \infty.$
\end{proof}

\begin{proof}[Proof of Theorem~\ref{thm:Approximation}]
 Due to Lemma~\ref{lem:bilipschitz}, we only need to show the convergence of the energies.  Let us pick a sequence $\varepsilon_j \downarrow 0$. As $f_{\varepsilon_j} \rightarrow f$ in $W^{\frac 3 2,2}$, we know that
 $\|f'_{\varepsilon_j}\| \rightarrow \|f'\|$ in measure. Hence, the integrand
 $$
  I_{\varepsilon_j} (x,y) = \left(\frac 1 {\|f_{\varepsilon_j} (x) - f_{\varepsilon_j}(y)\|^2} - \frac 1 {d_{f_{\varepsilon_j}}(x,y)^2} \right) \|f_{\varepsilon_j}'(x)\| \|f_{\varepsilon_j}'(y)\|
 $$
 converges to 
 $$
  I (x,y) = \left(\frac 1 {\|f (x) - f(y)\|^2} - \frac 1 {d_{f}(x,y)^2} \right) \|f'(x)\| \|f'(y)\|
 $$
 in measure. We will now complete the proof by showing that the integrands $I_{\varepsilon_j} (x,y)$ are uniformly integrable, as then an application of Vitali's theorem yields the convergence  of the energies.
 
 For this purpose, we estimate  setting $ w = x - y $,
 \begin{align*}
   I_\varepsilon (x,y)& = \left(\frac 1 {\|f_\varepsilon (x) - f_\varepsilon(y)\|^2} - \frac 1 {d_{f_\varepsilon}(x,y)^2} \right) \|f_\varepsilon'(x)\| \|f_\varepsilon'(y)\|
   \\
   &\leq C  \frac { d_{f_\varepsilon}(x,y)^2 - \|f_\varepsilon(x) - f_\varepsilon(y)\|^2}{|x-y|^4} \\
   & \leq C \frac { \int_{0}^1 \int_0^1 (\|f_\varepsilon'(x+\sigma_1w)\| \|f_\varepsilon'(x+\sigma_2 w)\| - f_\varepsilon'(x+\sigma_1 w) f_\varepsilon'(x+\sigma_2w)) d \sigma_1 d \sigma_2 }{|x-y|^2} \\
   & = \frac C 2  \frac {
   \int_{0}^1 \int_0^1 \left(\|f_\varepsilon'(x+\sigma_1w)\| \|f_\varepsilon'(x+\sigma_2 w)\| \left\| \frac { f_\varepsilon'(x+\sigma_1 w) } {\|f_\varepsilon'(x+\sigma_1 w)\|} - 
   \frac {f_\varepsilon'(x+\sigma_2w)} {\| f_\varepsilon'(x+\sigma_2 w)} \right\|^2 \right)d \sigma_1 d \sigma_2 }{|x-y|^2} \\
   & \leq C \frac { \int_{0}^1 \int_0^1 \left\|  f_\varepsilon'(x+\sigma_1 w) -f_\varepsilon'(x+\sigma_2w) \right\|^2 d \sigma_1 d \sigma_2 }{|x-y|^2}  =: C \tilde I_\varepsilon(x,y).
 \end{align*}
We get by first applying a binomial formula and then Cauchy-Schwartz
\begin{align*}
 &\int_{\mathbb R / \mathbb Z} \int_{\mathbb R / \mathbb Z} \left| \frac { \int_{0}^1 \int_0^1 \|  f_\varepsilon'(x+\sigma_1 w -f_\varepsilon'(x+\sigma_2w) \|^2 - \|f'(x+\sigma_1 w) -f'(x+\sigma_2w)\|^2d \sigma_1 d \sigma_2  }{|x-y|^2} \right| dx dy
 \\ 
 &= \int_{\mathbb R / \mathbb Z} \int_{\mathbb R / \mathbb Z} \Bigg| \frac { \int_{0}^1 \int_0^1 (  f_\varepsilon'(x+\sigma_1 w) -f_\varepsilon'(x+\sigma_2w)) + (f'(x+\sigma_1 w -f'(x+\sigma_2w))}{|x-y|} \\ &
 \frac{ ( f_\varepsilon'(x+\sigma_1 w) -f_\varepsilon'(x+\sigma_2w)) - (f'(x+\sigma_1 w) -f'(x+\sigma_2w))d \sigma_1 d \sigma_2 }{|x-y|} \Bigg| dx dy
 \\
 & \leq \left(\int_{\mathbb R / \mathbb Z}  \int_{\mathbb R / \mathbb Z} \frac { \int_{0}^1 \int_0^1 \|(  f_\varepsilon'(x+\sigma_1 w) -f_\varepsilon'(x+\sigma_2w)) + (f'(x+\sigma_1 w) -f'(x+\sigma_2w))\|^2 d \sigma_1 d \sigma_2 }{|x-y|^2} dx dy \right)^{\frac 1 2} 
 \\ & \left( \int_{\mathbb R / \mathbb Z} \int_{\mathbb R / \mathbb Z} \frac { \int_{0}^1 \int_0^1\| (  f_\varepsilon'(x+\sigma_1 w) -f_\varepsilon'(x+\sigma_2w)) - (f'(x+\sigma_1 w) -f'(x+\sigma_2w))\|^2 d \sigma_1 d \sigma_2 }{|x-y|^2} dx dy\right)^{\frac 1 2} 
 \\ & \leq C (\|f_\varepsilon\|_{W^{3/2,2}} + \|f\|_{W^{3/2,2}})
 \|f- f_\varepsilon\|_{W^{\frac 32,2}} \rightarrow 0.
\end{align*}
So for a sequence $\varepsilon_j \downarrow 0$ the integrands $\tilde I_{\varepsilon_j}$ are uniformly integrable as they converge to $$\tilde I (x,y) := \frac { \int_{0}^1 \int_0^1 \left\|  f'(x+\sigma_1 w) -f'(x+\sigma_2w) \right\|^2 d \sigma_1 d \sigma_2 }{|x-y|^2} $$ in $L^1$. As $I_\varepsilon \leq C \tilde I_\varepsilon$, also the integrands $I_{\varepsilon_j}$ are uniformly integrable.
\end{proof}

\section{$\Gamma$-Convergence}  \label{sec:Gamma}

To formulate our $\Gamma$-convergence result, we consider polygons as piecewise linear maps. We call a map $p:\mathbb R / \mathbb Z \rightarrow \mathbb R^n$  a \emph{closed polygon} with the $m$ vertices $p(\theta_i) \in \mathbb R^n,$ $i=1,\ldots, m$ if there are points $\theta_i \in [0,1)$, $\theta_1<\theta_2< \cdots < \theta_m$ such that $p$ is linear between two neighboring points $\theta_i$ and $\theta_{i+1}$, $i=1, \ldots, m$, or $\theta_m$ and $\theta_1$. We denote by
$$
 \delta(p) = \max\{|\theta_{i+1}-\theta_i|: i =1, \ldots,m \}
$$
the \emph{fineness} of the discretization, where $\theta_{m+1} = 1 + \theta_1.$
Let $P_m \supset W^{1,q}$, $\forall q \in [1, \infty]$ denote the space of all closed polygons.

Unfortunately there is one real obstacle to $\Gamma$-convergence, which is the fact that we have to explicitly assume that the fineness of the polygons is going to $0$ in order to obtain the $\liminf$-inequality. Which however is a very natural assumption to make.

\begin{theorem} Let $q \in [1,\infty).$ Then the discretized M\"obius energy $E_{cos}^{m}$ conditionally $\Gamma$-converges to the energy $E_{\cos}= E-4$ defined on the space of all regular curves in $C^{0,1}(\mathbb R / \mathbb Z, \mathbb R^n)$ with respect to the $W^{1,q}$-topology
in the following sense:
\begin{enumerate}
 \item ($\liminf$-inequality) For a sequence $p_m \in P_m$, $m \in \mathbb N$, converging to $f \in W^{1,q}$ in $W^{1,q}$ we have
 $$
  E_{\cos}(f) \leq \liminf_{m \rightarrow \infty} E_{\cos}^m(p_m)
 $$
 if the fineness of the discretization $\delta (f_m)$ converges to $0$ as $m$ goes to $\infty.$
 \item ($\limsup$-inequality) For all regular curves $f \in C^{0,1}(\mathbb R / \mathbb Z, \mathbb R^n)$ there is a sequence $p_m \in P_m$ such that
 $$
  E_{\cos}(f) \geq \limsup_{m \rightarrow \infty} E_{\cos}^m(p_m).
 $$
\end{enumerate}
\end{theorem}

Let us prove the $\liminf$- and $\limsup$-inequalities stated above in the following subsections.

\subsection{Proof of the \protect{$\liminf$}-Inequality} Let $p_m  \in P_m$, $m \in N$, be a sequence of polygons converging to $f$ in $W^{1,q}$. Then $p_m$ convergese uniformly to $f$ and after going to a subsequence, we can assume that $p_m'$ converges pointwise almost everywhere to $f$.

We rewrite the discrete energy as an integral leaving the superscript $m$ in the intermediate steps 
\begin{align*}
 E^{m}_{cos}(p_m) &=  \sum_{d_m(i,j) > 1} \frac {\|\Delta_i p_m\|\cdot \|\Delta_j p_m\|}{\|\Delta_i^j p_m\| \cdot \|\Delta_{i+1}^{j+1} p_m\|} \left( 1 - \frac 1 2 \left(\cos (\alpha^m_{ij}) + \cos (\tilde \alpha^m_{ij}) \right) \right)
 = \int_{\mathbb R / \mathbb Z} I^m_{p_m}(x,y) dx dy
\end{align*}
where 
\begin{align*}
 I_{p_m}^m (x,y) & = \sum_{d_m(i,j)>1} \frac {\left( 1 - \frac 1 2 \left(\cos (\alpha^m_{ij}) + \cos (\tilde \alpha^m_{ij}) \right) \right)} {\|\Delta_i^j p_m\| \cdot \|\Delta_{i+1}^{j+1} p_m\|}
 \frac{\|\Delta_i p_m\|}{|\theta_{i+1} - \theta_i|} \frac{\|\Delta_j p_m\|}{|\theta_{j+1}-\theta_j|} \chi_{[\theta_i,\theta_{i+1}]} \chi_{[\theta_j,\theta_{j+1}]} \\
 & = \sum_{d_m(i,j)>1} \frac { 1 - \frac 1 2 \left(\cos (\alpha^m_{ij}) + \cos (\tilde \alpha^m_{ij}) \right)} {\|\Delta_i^j p_m\| \cdot \|\Delta_{i+1}^{j+1} p_m\|}
 \|(p_m)'(x)\| \cdot \|(p_m)'(y)\| \chi_{[\theta_i,\theta_{i+1}]} \chi_{[\theta_j,\theta_{j+1}]}.
\end{align*}
As $I_{p_m}^m (x,y)$ converges pointwise almost everywhere to $\frac {1-\cos \alpha_f(x,y)} {\|f(x) - f(y)\|^2} \|f'(x)\| \|f'(y)\|$ if the fineness goes to $0$ we get from Fatou's lemma
\begin{equation*}
 \liminf_{m \rightarrow \infty} E^{m}_{\cos}(p_m) \geq E_{\cos}(f). 
\end{equation*}

\subsection{Proof of the \protect{$\limsup$}-Inequality}

\subsubsection{The  \protect{$\limsup$}-inequality for \protect{$C^{1,1}$} curves}
To prove the $\limsup$-inequality, we start with a regular curve $f \in W^{\frac 32,2}(\mathbb R / \mathbb Z, \mathbb R^n) \cap W^{1, \infty}(\mathbb R / \mathbb Z, \mathbb R^n)$. 

Let us first assume  that $f \in C^{1,1}$. We consider the polygons $p_m$ with edges $p_m(\theta_k)=f(\theta_k)$ for $k=1, \ldots, m$, where $\theta_k = \frac k m$. Let us remind that we assume that $p_m$ is the linear interpolation between these vertices.

The result is based on the following lemma.
\begin{lemma} \label{lem:uniformbound}
There is constant $C<\infty$ depending only of $f$ such that 
\begin{equation*}
   1 - \frac 1 2 \left(\cos (\alpha_{ij}) + \cos (\tilde \alpha_{ij}) \right)  \leq C \left( d_m(i,j)^2 m^{-2} \right)
\end{equation*}
where $d_m(i,j) = \min\{|i-j|,|i+m-j| \}$.
\end{lemma}

Before proving the lemma, let us show how it can be used to prove the $\limsup$-inequality. Lemma~\ref{lem:uniformbound} together with the fact that $f$ is bi-Lipschitz tells us that the integrands are uniformly bounded. Furthermore, they converge pointwise to 
$$
\frac {1-\cos \alpha_f(x,y)} {\|f(x) - f(y)\|^2} \|f'(x)\| \|f'(y)\|.
$$
By Lebesgue's theorem of dominated convergence, we hence get 
$$
 \lim_{m \rightarrow \infty} E^{m}_{\cos}(p_m) = E_{\cos}(f).
$$

We will base the proof of Lemma \ref{lem:uniformbound} on the fact that the global curvature of a regular injective curves is bounded if and only if the curve is of class $C^{1,1}$. This fact was independently found and proven in \cite{Gonzalez2002} and \cite{Cantarella2002}.

\begin{proof}[Proof of \protect{Lemma~\ref{lem:uniformbound}}]

Using the symmetry of the problem, it is enough to show that 
 \begin{equation*}
  1- \cos \alpha_{ij} = |i-j|^2 O (m^{-2}).
 \end{equation*}
We start with showing a lower bound for the radii of the circles $C_{i,j}$.

For three points $x,y,z \in \mathbb R / \mathbb Z$ we denote by
$$
 r_{f}(x,y,z)
$$
the radius of the circumcircle of the three points $f(x), f(y), f(z)$ if the three points are not co-linear. Otherwise, we set $r(x,y,z)=\infty$.
The curvature of this circumcircle
$$
 \kappa_{f}(x,y,z) := \frac 1 {r_{f}(x,y,z)}
$$
is also known as \emph{Menger curvature}. It was proven in in \cite{Gonzalez2002} and \cite{Cantarella2002} that the \emph{global curvature}
$$
 \kappa_{f}:=\sup_{x,y,z} \kappa_f (x,y,z)
$$
is bounded if and only if $f$ is a $C^{1,1}$ curve. So in our case
\begin{equation} \label{eq:LowerBoundRadius}
 \inf_{x,y,z} r_f(x,y,z) >0.
\end{equation}

%

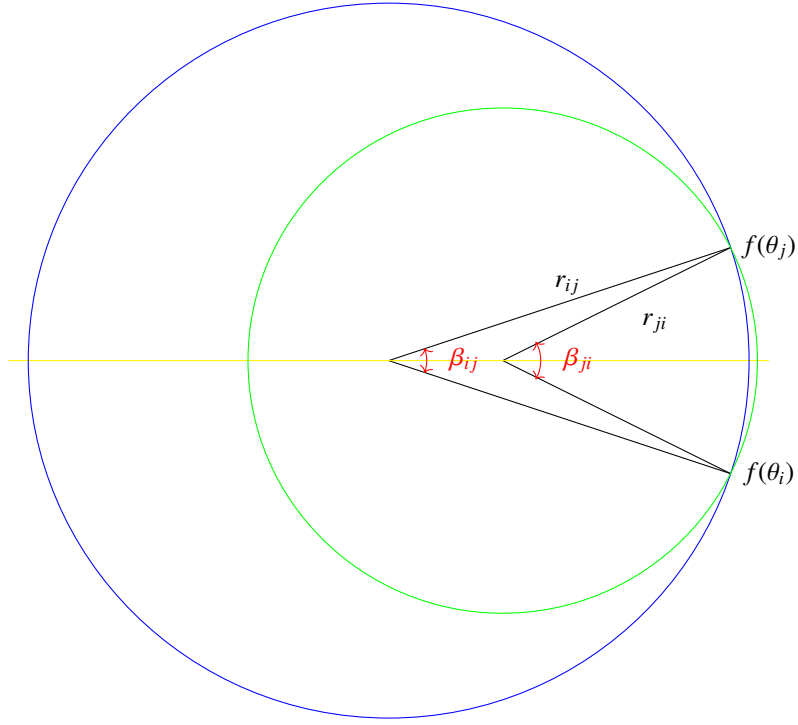
\begin{figure} \label{fig:estAngle}
\begin{tikzpicture}[scale=.5,rotate=90]
\coordinate (C) at (0,0) ;
    \coordinate (A) at (0,4) ;
    \coordinate (B) at (0,-4) ;
\draw (3,6) -- (0,-3);
\draw (3,6) -- (6,-3);
\draw (3,3) -- (0,-3);
\draw (3,3) -- (6,-3);
\draw [yellow] (3,16) -- (3,-4);

\node at (5,1.3) {$r_{ij}$};
\node at (4,-1) {$r_{ji}$};
\draw [blue] (3,6) circle (9.48);
\draw [green] (3,3) circle (6.70);
\draw[red,->] (3,2) arc (-90:-60:1);
\draw[red,->] (3,2) arc (-90:-120:1);

\draw[red,->] (3,5) arc (-90:-70.5:1);
\draw[red,->] (3,5) arc (-90:-109.47:1);

\node[red] at (3,4) {$\beta_{ij}$};
\node[red] at (3,1) {$\beta_{ji}$};

\node at (0,-4) {$f(\theta_i)$};
\node at (6,-4) {$f(\theta_j)$};

\end{tikzpicture}
\caption{This picture illustrates the estimate of the two angles $\beta_{i,j}$ and $\beta_{j,i}$. Note, that the two indicated circles might not lie in one plane angle. Neverthelsess. the anlge between them and the line connecting the two points  $f(\theta_i)$ and $f({\theta_j})$ is given by $\frac {\beta_{ij}} 2$ and $\frac {\beta_{ji}} 2$ respectively. We hence get the estimate $\alpha_{ij} \leq \frac {\beta_{ij} + \beta_{ji}} 2$ from the sub-additivity of the angles - which is nothing but the triangle inequality for the geodesic distance on $\mathbb S^2.$ }
\end{figure}

Now let us consider the two circles $C_{i,j}$ and $C_{j,i}$ and denote their radii by $r_{ij}$ and $r_{ji}$. Then these circles have the two points $f(\theta_i)$ and 
$f (\theta_j)$ in common. Let $0 \leq \beta_{i,j} \leq \pi$ denote the opening angle of the slice of the disk bounded by $C_{i,j}$ between the points $f(\theta_i)$ and $f(\theta_j)$. 
Let us assume w.l.o.g that $0 < \beta_{ji} \leq \beta_{ij} < \pi$. We have to keep in mind however that these circles might not belong to a single plane.

Note that nevertheless due to the fact that the angle between the line connecting $f(\theta_i)$ and $f(\theta_j)$ and the circles $C_{ij}$ and $C_{ji}$ is given by $\frac{\beta_{ij}}2$ and $\frac {\beta_{ji}}2$ resp., we get
$$
 \alpha_{ij} \leq \frac {\beta_{ij} + \beta_{ji}} 2 \leq \beta_{ij}.
$$
Since $\sin$ is monotonically increasing on $[0,  \pi / 2]$ this leads to
\begin{equation} \label{eq:estAngle}
 \sin(\alpha_{ij} /2)^2 \leq  \sin^2( \beta_{ij} / 2).
\end{equation}
We furthermore observe that   
\begin{equation*}
 \|f(\theta_i) - f(\theta_j)\| = 2 r_{ij}\sin\left(\frac {\beta_{ij}} 2\right)
\end{equation*}
and hence
\begin{equation*}
  \sin^2\left(\frac {\beta_{ij}}{2}\right) =\frac{ \|f(\theta_i) - f(\theta_j)\|^2 } {4 r_{i,j}^2}. 
\end{equation*}
Together with the estimate \eqref{eq:estAngle} this yields
\begin{equation*}
1- \cos(\alpha_{ij}) =  \sin^2\left(\frac {\alpha_{ij}}{2}\right) \leq \frac{ \|f(\theta_i) - f(\theta_j)\|^2 } {4 r_{i,j}^2} .
\end{equation*}
With the estimate \eqref{eq:LowerBoundRadius} and the fact that $\|f(\theta_i) - f(\theta_j)\|^2= d_m(i,j)^2 O(m^{-2})$  as $m\rightarrow \infty$ this proves the claim.
\end{proof}

\subsubsection{The \protect{$\limsup$}-inequality for \protect{$W^{\frac 32,2}$} curves } This follows literally in the same way as in \cite{Blatt2016b}, proof 2 of Theorem 4.8 .
Let us repeat the argument here for the convenience of the readers.

If  $f$ is a regular curves with bounded M\"obius energy, we can consider the smoothened curves $f_\varepsilon = f \ast \eta_\varepsilon$ and consider the curves 
By Theorem~\ref{thm:Approximation} we have
$\lim_{m \rightarrow \infty} E_{\cos}( f_{\frac 1 m}) = E_{\cos}(f).$

By the $\limsup$-inequality for $C^{1,1}$ curves, we can find in ${f}_{\frac 1m}$ inscribed equilateral $k$-gons $p_{m,k}$ with $\limsup_{k \rightarrow \infty} E_{\cos}^{k}(p_{m,k}) \leq E_{\cos}( f_{\frac 1m})$.
We observe that for all $k, \tilde m$ and $m'\geq \tilde m$ we have
$$ 
 \inf_{ m \geq \tilde m} E_{\cos}^k(p_{m,k}) \leq E_{\cos}^k(p_{m',k}).
$$
Taking the limes superior with respect to $k$ of this inequality we get
$$
 \limsup_{k\rightarrow \infty} \left(\inf_{m \geq \tilde m} E_{\cos}^k(p_{m,k}) \right) \leq \limsup_{k\rightarrow \infty} E_{\cos}^k(p_{m',k}) \leq E_{\cos}( f_{ \frac 1 {m'}})
$$
for all $m' \geq \tilde m$.
Hence,
$$
\limsup_{k\rightarrow \infty} \left(\inf_{m \geq \tilde m} E_{\cos}^k(p_{m,k}) \right) \leq \inf_{ m \geq \tilde m } 
E_{cos}( f_{\frac 1 m})\leq E_{\cos}(f)
$$
for all $\tilde m \in \mathbb N$.

If now for every $\tilde m, k\in \mathbb N$ we pick $m_{\tilde m, k}\in \mathbb N$ such that  $E_{\cos}^k(p_{m_{\tilde m , k}k})  \leq \inf_{m \geq \tilde m} E_{\cos}^k(p_{m,k})+ \frac 1k$ we get
$$
\limsup_{k\rightarrow \infty} E_{\cos}^{k}(p_{m_{\tilde m, k },k}) \leq \limsup_{k\rightarrow \infty} \left( \inf_{m \geq \tilde m} E_{\cos}^k(p_{m,k}) + \frac 1 k \right) =\limsup_{k\rightarrow \infty} \left( \inf_{m \geq \tilde m} E_{\cos}^k(p_{ m,k})  \right) \leq  E_{\cos}(f).
$$

Now we are ready to inductively define our sequence of polygons. We let $p_k$ be equal to $p_{m_{1,k},k}$ until $E_{\cos}^k(p_{m_{2,k}}) \leq E_{\cos}(f)+1$ for all bigger $k$. Then, we let $p_k$ be $p_{m_{2,k},k}$ until $E_{\cos}^k(p_{m_{3,k}}) \leq E_{\cos}(f) + \frac 1 2$ for all bigger $k$ and so on.  

This leads to a sequence $p_k$ of $k$-gons inscribed in the curves $ f_{\frac 1 {m_k}}$ such that both
$$
 \limsup_{m\rightarrow \infty} E_{\cos}^k(p_k) \leq E_{\cos}(f)
$$
and $m_k \rightarrow \infty$.

We finally have to prove that the polygons $p_k = p_{m_k, k}$ converge to $f$ in $W^ { 1,q}$ for $k\rightarrow \infty$ for all $q \in [1, \infty)$. By construction, we know that the $p_{m_k,k}$ are uniformly bounded in $W^{1,\infty}$. Let $p_{m_k,k}(x_1^{k})= f_{m_k}  (x_1^{k}),\ldots ,  p_{m_k,k} (x^{k}_k) = f_{m_k} (x_k^{k})$ be the vertices of the $k$-gon $p_{m_k,k}$. Due to the uniform bi-Lipschitz estimate we have $ \frac 1 {Ck} \leq |x_{i+1}^k - x_i^{k}| \leq \frac C k$ for a constant $C < \infty$ and hence we get for $x \in [x_i, x_{i+1}]$ using Taylor's approximation of first order in $x_i^k$ 
$$
 \| p_{k}(x) -   f_{\frac 1 {m_k }} (x)\| \leq C |x-x^k_i| \leq \frac C k.
$$
So, the $p_{k}$ converge uniformly to $f$.

To get the convergence of the derivatives, we use $p'_k = \overline{({f}'_{\frac 1 {m_k}})}_{[x^k_i,x^k_{i+1}]}$ for $x \in [x_i^k,x_{i+1}^k]$ to  estimate
\begin{align*}
 \int_{\mathbb R / \mathbb Z} \| p'_{k}-{f}'_{m_k} \| dx
 & = \sum_{i=1}^{k-1}\int_{x^k_i}^{x^k_{i+1}}\| \overline{({f}'_{m_k})}_{[x^k_i,x^k_{i+1}]}  - {f}'_{m_k} \| dx 
 \leq \sum_{i=1}^{k-1} |x^k_{i+1} -x^k_i| D_{ {f}'_{\frac 1 {m_k}}}(\frac C k)
 \\ & \leq C \cdot D_{ {f}_{\frac 1 {m_k}}}(\frac C k).
\end{align*}
From \eqref{eq:DBounded} we get $D_{{f}_{\frac 1 {m_k}}}  \leq  D_{f}  $.
So $D_{ {f}_{\frac 1 {m_k}}}( r)$ goes uniformly to zero as $r \rightarrow 0$ and hence $p_k'- { f}'_{\frac 1 {m_k}}$ converges to $0$ as $k$ goes to $\infty$. Since $f'_{\frac 1 {m_k}}$ converges to $f'$ in $L^1$, we deduce that $p'_{k}$ converges to $f'$ in $L^1$. 
Since the polygons are furthermore uniformly bounded in $W^{1, \infty}$, we get convergence in $W^{1,q}$ for all $q \in [1, \infty)$ 
from the interpolation estimate
$$
 \|f'\|^q_{L^q} \leq  \|f'\|_{L^\infty}^{q-1} \|f'\|^1_{L^1}.
$$

\section{Open Questions}

Though we get a quite satisfactory answer regarding $\Gamma$-convergence of our discretized energies, many elemetary problems in this context remain open. Even problems that can be answered in the context of other discretizations. 

Let us end this article by a short and by no means complete list of some interesting open problems:
 \begin{enumerate}
  \item What is the convergence rate of the inscribed polygons used in the above proof of the $\limsup$-inequality? Is it possible to drop the assumption that we are using an equipartition without losing the convergence?
  \item Do minimizers of the discrete energy within a given knot class exist? Due to the M\"obius invariance of the discretization the answer to this question is far from being obvious.
  \item Can one prove additional properties for minimizers, like a bi-Lipschitz estimate (of course depending on the fineness of the discretization) or some discrete versions of "higher regularity"?
  \item Is there a way to discretize the M\"obius energy in such a way that we preserve both the self-repulsiveness and the M\"obius invariance of the M\"obius energy?
 \end{enumerate}

\appendix

\section{A Formula for the Discretized Energy}

Let us derive a formula for our discretized energy for future reference. For this purpose, we
want to derive the expression of the circle passing through three points $ \mathrm{A}_1 $,
$ \mathrm{A}_2 $,
$ \mathrm{A}_3 $ in $ \mathbb{R}^n $.
Let $ \vecv_1 $,
$ \vecv_2 $,
$ \vecv_3 $ be their position vectors,
and let $ \vecc $,
and $ r $ be the position vector of center and the radius of the circle.
There exist $ t_1 $,
$ t_2 $,
$ t_3 \in \mathbb{R} $ such that
\[
	\vecc
	=
	t_1 \vecv_1 + t_2 \vecv_2 + t_3 \vecv_3
	,
	\quad
	t_1 + t_2 + t_3 =1 .
\]
We denote $ \angle \mathrm{A}_i = \varphi_i $.
Consider the triangle which vertices are $ \vecv_i $,
$ \vecv_j $,
$ \vecc $,
and its arcs $ S_k $,
where $ k \in \{ 1 , 2 , 3 \} $ such that $ k \ne i $ and $ k \ne j $.
We have
\[
	t_1 : t_2 : t_3 = S_1 : S_2 : S_3 .
\]
Since
\[
	S_i = \frac 12 r^2 \sin 2 \varphi_i ,
\]
we get
\[
	\vecc
	=
	\frac {
	( \sin 2 \varphi_1 ) \vecv_1
	+
	( \sin 2 \varphi_2 ) \vecv_2
	+
	( \sin 2 \varphi_3 ) \vecv_3
	}
	{ \sin 2 \varphi_1 + \sin 2 \varphi_2 + \sin 2 \varphi_3 } .
\]
By the law of sine it holds that
\[
	r
	= \frac { \| \vecv_1 - \vecv_2 \|_{ \mathbb{R}^n } } { 2 \sin \varphi_3 }
	= \frac { \| \vecv_2 - \vecv_3 \|_{ \mathbb{R}^n } } { 2 \sin \varphi_1 }
	= \frac { \| \vecv_3 - \vecv_1 \|_{ \mathbb{R}^n } } { 2 \sin \varphi_2 } .
\]
Therefore we have
\[
	r =
	\frac {
	\| \vecv_i - \vecv_j \|_{ \mathbb{R}^n }
	\| \vecv_j - \vecv_k \|_{ \mathbb{R}^n }
	\| \vecv_k - \vecv_i \|_{ \mathbb{R}^n }
	}
	{ 2 \| ( \vecv_i - \vecv_j ) \bigwedge ( \vecv_j - \vecv_k ) \|_{ \mathbb{R}^n } }
	,
\]
where $ i \ne j $,
$ j \ne k $,
$ k \ne i $.
Put
\[
	\vecu_i =
	\frac { \vecv_i - \vecc }
	{ \| \vecv_i - \vecc \|_{ \mathbb{R}^n } }
	=
	\frac { \vecv_i - \vecc } r
	,
\]
\[
	\vece_1 = \vecu_1 ,
	\quad
	\vece_2
	=
	\frac { \vecu_2 - ( \vecu_2 \cdot \vecu_1 ) \vecu_1 }
	{ \| \vecu_2 - ( \vecu_2 \cdot \vecu_1 ) \vecu_1 \|_{ \mathbb{R}^n } }
	,
\]
then the circle is given by
\[
	\vecx ( \theta )
	=
	\vecc
	+ r \left\{
	\left( \cos \frac \theta r \right) \vece_1
	+
	\left( \sin \frac \theta r \right) \vece_2
	\right\}
	.
\]
Here $ \theta \in \mathbb{R} / 2 \pi r \mathbb{Z} $ is an arc-length parameter of the circle.
Using $ \vecu_1 \cdot \vecu_2 = \cos 2 \varphi_3 $,
we have
\[
	r \{ \vecu_2 - ( \vecu_2 \cdot \vecu_1 ) \vecu_1 \}
	=
	\vecv_2 - \vecc - ( \cos 2 \varphi_3 ) ( \vecv_1 - \vecc )
	,
\]
\[
	\| \vecu_2 - ( \vecu_2 \cdot \vecu_1 ) \vecu_1 \|_{ \mathbb{R}^n }
	=
	\sqrt{ 1 - \cos^2 2 \varphi_3 }
	=
	\sin 2 \varphi_3
	.
\]
Therefore
\begin{align*}
	\vecx ( \theta )
	= & \
	\vecc
	+
	\left( \cos \frac \theta r \right) ( \vecv_1 - \vecc )
	+
	\left( \sin \frac \theta r \right)
	\frac
	{ \vecv_2 - \vecc - ( \cos 2 \varphi_3 ) ( \vecv_1 - \vecc ) }
	{ \sin 2 \varphi_3 }
	\\
	= & \
	\vecc
	+
	\frac 1 { \sin 2 \varphi_3 }
	\left[
	\left\{ \sin \left( 2 \varphi_3 - \frac \theta r \right) \right\} ( \vecv_1 - \vecc )
	+
	\left( \sin \frac \theta r \right) ( \vecv_2 - \vecc )
	\right]
	.
\end{align*}
It is easy to see
\[
	\vecx ( 0 ) = \vecv_1 ,
	\quad
	\vecx ( 2 r \varphi_3 ) = \vecv_2 ,
	\quad
	\vecx ( 2 r \varphi_3 + 2r \varphi_1 ) = \vecv_3
	.
\]
It follows from
\[
	\dot { \vecx } ( \theta )
	=
	\frac 1 { r \sin 2 \varphi_3 }
	\left[
	- \left\{ \cos \left( 2 \varphi_3 - \frac \theta r \right) \right\}
	( \vecv_1 - \vecc )
	+
	\left( \cos \frac \theta r \right) ( \vecv_2 - \vecc )
	\right]
\]
that
\[
	\dot { \vecx } (0)
	=
	\frac 1 { r \sin 2 \varphi_3 }
	\left\{
	- ( \cos 2 \varphi_3 )
	( \vecv_1 - \vecc )
	+
	\vecv_2 - \vecc
	\right\}
	.
\]
Using
\begin{align*}
	\vecv_1 - \vecc
	= & \
	t_2 ( \vecv_1 - \vecv_2 )
	+
	t_3 ( \vecv_1 - \vecv_3 )
	=
	\frac
	{
	( \sin 2 \varphi_2 ) ( \vecv_1 - \vecv_2 )
	+ 
	( \sin 2 \varphi_3 ) ( \vecv_1 - \vecv_3 )
	}
	{ \sin 2 \varphi_1 + \sin 2 \varphi_2 + \sin 2 \varphi_3 }
	,
	\\
	\vecv_2 - \vecc
	= & \
	t_1 ( \vecv_2 - \vecv_1 )
	+
	t_3 ( \vecv_2 - \vecv_3 )
	=
	\frac
	{
	( \sin 2 \varphi_1 ) ( \vecv_2 - \vecv_1 )
	+ 
	( \sin 2 \varphi_3 ) ( \vecv_2 - \vecv_3 )
	}
	{ \sin 2 \varphi_1 + \sin 2 \varphi_2 + \sin 2 \varphi_3 }
	,
\end{align*}
we obtain
\[
	\dot { \vecx } (0)
	=
	-
	\frac
	{
	( \sin^2 \varphi_2 ) ( \vecv_1 - \vecv_2 )
	+
	( \sin^2 \varphi_3 ) ( \vecv_3 - \vecv_1 )
	}
	{ r ( \sin \varphi_1 \cos \varphi_1 + \sin \varphi_2 \cos \varphi_2 + \sin \varphi_3 \cos \varphi_3 ) }
	.
\]
Using the notation
\[
	\Delta_i^j \vecv = \vecv_j - \vecv_i ,
\]
we have
\[
	\cos \varphi_i =
	\frac { \Delta_i^j \vecv } { \| \Delta_i^j \vecv \|_{ \mathbb{R}^n } }
	\cdot
	\frac { \Delta_i^k \vecv } { \| \Delta_i^k \vecv \|_{ \mathbb{R}^n } }
	,
	\quad
	\sin \varphi_i =
	\left\|
	\frac { \Delta_i^j \vecv } { \| \Delta_i^j \vecv \|_{ \mathbb{R}^n } }
	\bigwedge
	\frac { \Delta_i^k \vecv } { \| \Delta_i^k \vecv \|_{ \mathbb{R}^n } }
	\right\|_{ \mathbb{R}^n }
	,
\]
where $ i \ne j $,
$ j \ne k $,
$ k \ne i $.
Furthermore we have
\[
	r =
	\frac { \| \Delta_j^k \vecv \|_{ \mathbb{R}^n } }
	{
	2
	\left\|
	\frac { \Delta_i^j \vecv } { \| \Delta_i^j \vecv \|_{ \mathbb{R}^n } }
	\bigwedge
	\frac { \Delta_i^k \vecv } { \| \Delta_i^k \vecv \|_{ \mathbb{R}^n } }
	\right\|_{ \mathbb{R}^n }
	}
	,
\]
\[
	r \sin \varphi_i \cos \varphi_i
	=
	\frac 12
	\| \Delta_j^k \vecv \|_{ \mathbb{R}^n }
	\left(
	\frac { \Delta_i^j \vecv } { \| \Delta_i^j \vecv \|_{ \mathbb{R}^n } }
	\cdot
	\frac { \Delta_i^k \vecv } { \| \Delta_i^k \vecv \|_{ \mathbb{R}^n } }
	\right)
	.
\]
Consequently we have
\begin{align*}
	\dot { \vecx } (0)
	= & \
	-
	2
	\left\{
	\| \Delta_2^3 \vecv \|_{ \mathbb{R}^n }
	\left(
	\frac { \Delta_1^2 \vecv } { \| \Delta_1^2 \vecv \|_{ \mathbb{R}^n } }
	\cdot
	\frac { \Delta_1^3 \vecv } { \| \Delta_1^3 \vecv \|_{ \mathbb{R}^n } }
	\right)
	\right.
	\\
	& \quad \qquad
	\left.
	+ \,
	\| \Delta_3^1 \vecv \|_{ \mathbb{R}^n }
	\left(
	\frac { \Delta_2^3 \vecv } { \| \Delta_2^3 \vecv \|_{ \mathbb{R}^n } }
	\cdot
	\frac { \Delta_2^1 \vecv } { \| \Delta_2^1 \vecv \|_{ \mathbb{R}^n } }
	\right)
	\right.
	\\
	& \quad \qquad
	\left.
	+ \,
	\| \Delta_1^2 \vecv \|_{ \mathbb{R}^n }
	\left(
	\frac { \Delta_3^1 \vecv } { \| \Delta_3^1 \vecv \|_{ \mathbb{R}^n } }
	\cdot
	\frac { \Delta_3^2 \vecv } { \| \Delta_3^2 \vecv \|_{ \mathbb{R}^n } }
	\right)
	\right\}^{-1}
	\\
	& \quad \qquad
	\times
	\left\{
	\left\|
	\frac { \Delta_2^3 \vecv } { \| \Delta_2^3 \vecv \|_{ \mathbb{R}^n } }
	\bigwedge
	\frac { \Delta_2^1 \vecv } { \| \Delta_2^1 \vecv \|_{ \mathbb{R}^n } }
	\right\|_{ \mathbb{R}^n }^2
	\Delta_2^1 \vecv
	+
	\left\|
	\frac { \Delta_3^1 \vecv } { \| \Delta_3^1 \vecv \|_{ \mathbb{R}^n } }
	\bigwedge
	\frac { \Delta_3^2 \vecv } { \| \Delta_3^2 \vecv \|_{ \mathbb{R}^n } }
	\right\|_{ \mathbb{R}^n }^2
	\Delta_1^3 \vecv
	\right\}
	.
\end{align*}
Let $ \vecw_i $ and $ \overline \vecw_i $ be
\[
	\vecw_i = \Delta_i^{ i+1 } \vecv = \vecv_{ i+1 } - \vecv_i
	,
	\quad
	\overline { \vecw }_i = \frac { \vecw_i } { \| \vecw_i \|_{ \mathbb{R}^n } }
	.
\]
Then
\begin{align*}
	\dot { \vecx } (0)
	= & \
	-
	\frac
	{
	2 \left(
	\left\| \overline{ \vecw }_2 \wedge \overline{ \vecw }_1 \right\|_{ \mathbb{R}^n }^2
	\vecw_1
	+
	\left\| \overline{ \vecw }_2 \wedge \overline{ \vecw }_3 \right\|_{ \mathbb{R}^n }^2
	\vecw_3
	\right)
	}
	{
	\| \vecw_1 \|_{ \mathbb{R}^n }
	( \overline{ \vecw }_2 \cdot \overline{ \vecw }_3 )
	+
	\| \vecw_2 \|_{ \mathbb{R}^n }
	( \overline{ \vecw }_3 \cdot \overline{ \vecw }_1 )
	+
	\| \vecw_3 \|_{ \mathbb{R}^n }
	( \overline{ \vecw }_1 \cdot \overline{ \vecw }_2 )
	}
	\\
	= & \
	-
	\frac 2
	{
	\| \vecw_1 \|_{ \mathbb{R}^n }
	\| \vecw_2 \|_{ \mathbb{R}^n }
	\| \vecw_3 \|_{ \mathbb{R}^n }
	}
	\\
	& \quad
	\times
	\frac
	{
	\| \vecw_3 \|_{ \mathbb{R}^n }^2
	\left\| \vecw_2 \wedge \vecw_1 \right\|_{ \mathbb{R}^n }^2
	\vecw_1
	+
	\| \vecw_1 \|_{ \mathbb{R}^n }^2
	\left\| \vecw_2 \wedge \vecw_3 \right\|_{ \mathbb{R}^n }^2
	\vecw_3
	}
	{
	\| \vecw_1 \|_{ \mathbb{R}^n }^2
	( \vecw_2 \cdot \vecw_3 )
	+
	\| \vecw_2 \|_{ \mathbb{R}^n }^2
	( \vecw_3 \cdot \vecw_1 )
	+
	\| \vecw_3 \|_{ \mathbb{R}^n }^2
	( \vecw_1 \cdot \vecw_2 )
	}
\end{align*}
Using $ \vecw_1 + \vecw_2 + \vecw_3 = \veco $,
we have
\[
	\vecw_2 = - \vecw_1 - \vecw_3
	,
\]
\[
	\vecw_2 \wedge \vecw_3
	=
	-
	\vecw_2 \wedge \vecw_1
	,
\]
and
\begin{align*}
	&
	\| \vecw_1 \|_{ \mathbb{R}^n }^2
	( \vecw_2 \cdot \vecw_3 )
	+
	\| \vecw_2 \|_{ \mathbb{R}^n }^2
	( \vecw_3 \cdot \vecw_1 )
	+
	\| \vecw_3 \|_{ \mathbb{R}^n }^2
	( \vecw_1 \cdot \vecw_2 )
	\\
	& \quad
	=
	\| \vecw_1 \|_{ \mathbb{R}^n }^2
	\{ \vecw_2 \cdot ( - \vecw_1 - \vecw_3 ) \}
	+
	\| \vecw_2 \|_{ \mathbb{R}^n }^2
	\{ ( - \vecw_1 - \vecw_2 ) \cdot \vecw_1 \}
	\\
	& \quad \qquad
	+ \,
	\| - \vecw_1 - \vecw_2 \|_{ \mathbb{R}^n }^2
	( \vecw_1 \cdot \vecw_2 )
	\\
	& \quad
	=
	-
	\| \vecw_1 \|_{ \mathbb{R}^n }^2
	( \vecw_2 \cdot \vecw_1 )
	-
	\| \vecw_1 \|_{ \mathbb{R}^n }^2
	\| \vecw_2 \|_{ \mathbb{R}^n }^2
	-
	\| \vecw_2 \|_{ \mathbb{R}^n }^2
	\| \vecw_1 \|_{ \mathbb{R}^n }^2
	-
	\| \vecw_2 \|_{ \mathbb{R}^n }^2
	( \vecw_2 \cdot \vecw_1 )
	\\
	& \quad \qquad
	+ \,
	\left\{
	\| \vecw_1 \|_{ \mathbb{R}^n }^2
	+ 2 ( \vecw_1 \cdot \vecw_2 )
	+ \| \vecw_2 \|_{ \mathbb{R}^n }^2
	\right\}
	( \vecw_1 \cdot \vecw_2 )
	\\
	& \quad
	=
	- 2
	\left\{
	\| \vecw_1 \|_{ \mathbb{R}^n }^2
	\| \vecw_2 \|_{ \mathbb{R}^n }^2
	-
	( \vecw_1 \cdot \vecw_2 )^2
	\right\}
	\\
	& \quad
	=
	- 2
	\| \vecw_1 \wedge \vecw_2 \|_{ \mathbb{R}^n }^2
	.
\end{align*}
Hence
\[
	\vecx (0)
	=
	\frac
	{
	\| \vecw_3 \|_{ \mathbb{R}^n }^2
	\vecw_1
	+
	\| \vecw_1 \|_{ \mathbb{R}^n }^2
	\vecw_3
	}
	{
	\| \vecw_1 \|_{ \mathbb{R}^n }
	\| \vecw_2 \|_{ \mathbb{R}^n }
	\| \vecw_3 \|_{ \mathbb{R}^n }
	}
	=
	\frac
	{
	\| \Delta_3 \vecv \|_{ \mathbb{R}^n }^2
	\Delta_1 \vecv
	+
	\| \Delta_1 \vecv \|_{ \mathbb{R}^n }^2
	\Delta_3 \vecv
	}
	{
	\| \Delta_1 \vecv \|_{ \mathbb{R}^n }
	\| \Delta_2 \vecv \|_{ \mathbb{R}^n }
	\| \Delta_3 \vecv \|_{ \mathbb{R}^n }
	}
	.
\]
Generally we obtain
\begin{lemma} \label{lem:formulatangent}
Let $ C $ be the circle passing through three end points of position vectors $ \vecv_1 $,
$ \vecv_2 $ and $ \vecv_3 \in \mathbb{R}^n $.
Then the unit tangent vector $ \vect_i $ at the end point of $ \vecv_i $ is
\[
	\vect_i
	=
	\frac
	{
	\| \Delta_{ i-1 } \vecv \|_{ \mathbb{R}^n }^2
	\Delta_i \vecv
	+
	\| \Delta_i \vecv \|_{ \mathbb{R}^n }^2
	\Delta_{ i-1 } \vecv
	}
	{
	\| \Delta_i \vecv \|_{ \mathbb{R}^n }
	\| \Delta_{ i+1 } \vecv \|_{ \mathbb{R}^n }
	\| \Delta_{ i-1 } \vecv \|_{ \mathbb{R}^n }
	}
	.
\]
\end{lemma}
\par
From lemma~\ref{lem:formulatangent},
the unit tangent vector $ \vect ( i , i+1, j ) $ at $ \vecf ( \theta_j ) $ of the circle passing though $ \vecf ( \theta_i ) $,
$ \vecf ( \theta_{ i+1 } ) $,
$ \vecf ( \theta_j ) $ is
\begin{equation} \label{eq:tii+1j}
\begin{aligned}
	\vect ( i , i+1 , j )
	= & \
	\frac
	{
	\| \Delta_{ i+1 }^j \vecf \|_{ \mathbb{R}^n }^2
	\Delta_j^i \vecf
	+
	\| \Delta_j^i \vecf \|_{ \mathbb{R}^n }^2
	\Delta_{ i+1 }^j \vecf
	}
	{
	\| \Delta_j^i \vecf \|_{ \mathbb{R}^n }
	\| \Delta_i \vecf \|_{ \mathbb{R}^n }
	\| \Delta_{ i+1 }^j \vecf \|_{ \mathbb{R}^n }
	}
	\\
	= & \
	\frac
	{
	-
	\| \Delta_{ i+1 }^j \vecf \|_{ \mathbb{R}^n }^2
	\Delta_i^j \vecf
	+
	\| \Delta_i^j \vecf \|_{ \mathbb{R}^n }^2
	\Delta_{ i+1 }^j \vecf
	}
	{
	\| \Delta_i^j \vecf \|_{ \mathbb{R}^n }
	\| \Delta_i \vecf \|_{ \mathbb{R}^n }
	\| \Delta_{ i+1 }^j \vecf \|_{ \mathbb{R}^n }
	}
	\\
	= & \
	\frac
	{
	-
	\left(
	\| \Delta_{ i+1 }^j \vecf \|_{ \mathbb{R}^n }^2
	-
	\| \Delta_i^j \vecf \|_{ \mathbb{R}^n }^2
	\right)
	\Delta_i^j \vecf
	-
	\| \Delta_i^j \vecf \|_{ \mathbb{R}^n }^2
	\Delta_i \vecf
	}
	{
	\| \Delta_i^j \vecf \|_{ \mathbb{R}^n }
	\| \Delta_i \vecf \|_{ \mathbb{R}^n }
	\| \Delta_{ i+1 }^j \vecf \|_{ \mathbb{R}^n }
	}
	\\
	= & \
	\frac
	{
	\left\{ \Delta_i \vecf \cdot \left( 2 \Delta_i^j \vecf - \Delta_i \vecf \right) \right\}
	\Delta_i^j \vecf
	-
	\| \Delta_i^j \vecf \|_{ \mathbb{R}^n }^2
	\Delta_i \vecf
	}
	{
	\| \Delta_i^j \vecf \|_{ \mathbb{R}^n }
	\| \Delta_i \vecf \|_{ \mathbb{R}^n }
	\| \Delta_{ i+1 }^j \vecf \|_{ \mathbb{R}^n }
	}
	\\
	= & \
	\frac 1
	{
	\| \Delta_i^j \vecf \|_{ \mathbb{R}^n }
	\| \Delta_{ i+1 }^j \vecf \|_{ \mathbb{R}^n }
	}
	\left[
	\left\{ \frac { \Delta_i \vecf } { \| \Delta_i \vecf \|_{ \mathbb{R}^n } }
	\cdot \left( 2 \Delta_i^j \vecf - \Delta_i \vecf \right) \right\}
	\Delta_i^j \vecf
	-
	\| \Delta_i^j \vecf \|_{ \mathbb{R}^n }^2
	\frac { \Delta_i \vecf } { \| \Delta_i \vecf \|_{ \mathbb{R}^n } }
	\right]
	\\
	= & \
	\frac 1
	{
	\| \Delta_i^j \vecf \|_{ \mathbb{R}^n }
	\| \Delta_{ i+1 }^j \vecf \|_{ \mathbb{R}^n }
	}
	\\
	& \times
	\left\{
	\left(
	2 \Delta_i^j \vecf \cdot \frac { \Delta_i \vecf } { \| \Delta_i \vecf \|_{ \mathbb{R}^n } }
	\right)
	\Delta_i^j \vecf
	-
	\| \Delta_i \vecf \|_{ \mathbb{R}^n }
	\Delta_i^j \vecf
	-
	\| \Delta_i^j \vecf \|_{ \mathbb{R}^n }^2
	\frac { \Delta_i \vecf } { \| \Delta_i \vecf \|_{ \mathbb{R}^n } }
	\right\}
	.
\end{aligned}
\end{equation}
The unit tangent vector $ \vect ( j+1, i , j ) $ at $ \vecf ( \theta_j ) $ of the circle passing though $ \vecf ( \theta_j ) $,
$ \vecf ( \theta_{ j+1 } ) $,
$ \vecf ( \theta_i ) $ is
\begin{equation} \label{eq:j+1ij}
\begin{aligned}
	\vect ( j+1 , i , j )
	= & \
	\frac
	{
	\| \Delta_i^j \vecf \|_{ \mathbb{R}^n }^2
	\Delta_j \vecf
	+
	\| \Delta_j \vecf \|_{ \mathbb{R}^n }^2
	\Delta_i^j \vecf
	}
	{
	\| \Delta_j \vecf \|_{ \mathbb{R}^n }
	\| \Delta_{ j+1 }^i \vecf \|_{ \mathbb{R}^n }
	\| \Delta_i^j \vecf \|_{ \mathbb{R}^n }
	}
	\\
	= & \
	\frac
	{
	\| \Delta_i^j \vecf \|_{ \mathbb{R}^n }^2
	\Delta_j \vecf
	+
	\| \Delta_j \vecf \|_{ \mathbb{R}^n }^2
	\Delta_i^j \vecf
	}
	{
	\| \Delta_j \vecf \|_{ \mathbb{R}^n }
	\| \Delta_i^{ j+1 } \vecf \|_{ \mathbb{R}^n }
	\| \Delta_i^j \vecf \|_{ \mathbb{R}^n }
	}
	\\
	= & \
	\frac 1
	{
	\| \Delta_i^j \vecf \|_{ \mathbb{R}^n }
	\| \Delta_i^{ j+1 } \vecf \|_{ \mathbb{R}^n }
	}
	\left(
	\| \Delta_i^j \vecf \|_{ \mathbb{R}^n }^2
	\frac { \Delta_j \vecf } { \| \Delta_j \vecf \|_{ \mathbb{R}^n } }
	+
	\| \Delta_j \vecf \|_{ \mathbb{R}^n }
	\Delta_i^j \vecf
	\right)
	.
\end{aligned}
\end{equation}
Consequently
\begin{align*}
	\cos \alpha_{ij}
	= & \
	\vect ( i , i+1 , j ) \cdot \vect ( j+1 , i , j )
	\\
	= & \
	\frac 1
	{
	\| \Delta_i^j \vecf \|_{ \mathbb{R}^n }
	\| \Delta_{ i+1 }^j \vecf \|_{ \mathbb{R}^n }
	}
	\\
	& \quad \qquad
	\times
	\left\{
	\left(
	2 \Delta_i^j \vecf \cdot \frac { \Delta_i \vecf } { \| \Delta_i \vecf \|_{ \mathbb{R}^n } }
	\right)
	\Delta_i^j \vecf
	-
	\| \Delta_i \vecf \|_{ \mathbb{R}^n }
	\Delta_i^j \vecf
	-
	\| \Delta_i^j \vecf \|_{ \mathbb{R}^n }^2
	\frac { \Delta_i \vecf } { \| \Delta_i \vecf \|_{ \mathbb{R}^n } }
	\right\}
	\\
	& \quad
	\cdot
	\frac 1
	{
	\| \Delta_i^j \vecf \|_{ \mathbb{R}^n }
	\| \Delta_i^{ j+1 } \vecf \|_{ \mathbb{R}^n }
	}
	\left(
	\| \Delta_i^j \vecf \|_{ \mathbb{R}^n }^2
	\frac { \Delta_j \vecf } { \| \Delta_j \vecf \|_{ \mathbb{R}^n } }
	+
	\| \Delta_j \vecf \|_{ \mathbb{R}^n }
	\Delta_i^j \vecf
	\right)
	\\
	= & \
	\frac 1
	{
	\| \Delta_{ i+1 }^j \vecf \|_{ \mathbb{R}^n }
	\| \Delta_i^{ j+1 } \vecf \|_{ \mathbb{R}^n }
	}
	\left\{
	2
	\left(
	\Delta_i^j \vecf \cdot \frac { \Delta_i \vecf } { \| \Delta_i \vecf \|_{ \mathbb{R}^n } }
	\right)
	\left(
	\Delta_i^j \vecf \cdot \frac { \Delta_j \vecf } { \| \Delta_j \vecf \|_{ \mathbb{R}^n } }
	\right)
	\right.
	\\
	& \quad \quad
	\left.
	- \,
	\| \Delta_i \vecf \|_{ \mathbb{R}^n }
	\left(
	\Delta_i^j \vecf \cdot \frac { \Delta_j \vecf } { \| \Delta_j \vecf \|_{ \mathbb{R}^n } }
	\right)
	-
	\| \Delta_i^j \vecf \|_{ \mathbb{R}^n }^2
	\frac { \Delta_i \vecf } { \| \Delta_i \vecf \|_{ \mathbb{R}^n } }
	\cdot
	\frac { \Delta_j \vecf } { \| \Delta_j \vecf \|_{ \mathbb{R}^n } }
	\right.
	\\
	& \quad \quad
	\left.
	+ \,
	2
	\| \Delta_j \vecf \|_{ \mathbb{R}^n }
	\left(
	\Delta_i^j \vecf \cdot \frac { \Delta_i \vecf } { \| \Delta_i \vecf \|_{ \mathbb{R}^n } }
	\right)
	-
	\| \Delta_i \vecf \|_{ \mathbb{R}^n }
	\| \Delta_j \vecf \|_{ \mathbb{R}^n }
	\right.
	\\
	& \quad \quad
	\left.
	- \,
	\| \Delta_j \vecf \|_{ \mathbb{R}^n }
	\left(
	\Delta_i^j \vecf \cdot \frac { \Delta_i \vecf } { \| \Delta_i \vecf \|_{ \mathbb{R}^n } }
	\right)
	\right\}
	\\
	= & \
	\frac 1
	{
	\| \Delta_{ i+1 }^j \vecf \|_{ \mathbb{R}^n }
	\| \Delta_i^{ j+1 } \vecf \|_{ \mathbb{R}^n }
	}
	\left\{
	2
	\left(
	\Delta_i^j \vecf \cdot \frac { \Delta_i \vecf } { \| \Delta_i \vecf \|_{ \mathbb{R}^n } }
	\right)
	\left(
	\Delta_i^j \vecf \cdot \frac { \Delta_j \vecf } { \| \Delta_j \vecf \|_{ \mathbb{R}^n } }
	\right)
	\right.
	\\
	& \quad \quad
	\left.
	- \,
	\| \Delta_i \vecf \|_{ \mathbb{R}^n }
	\left(
	\Delta_i^j \vecf \cdot \frac { \Delta_j \vecf } { \| \Delta_j \vecf \|_{ \mathbb{R}^n } }
	\right)
	-
	\| \Delta_i^j \vecf \|_{ \mathbb{R}^n }^2
	\frac { \Delta_i \vecf } { \| \Delta_i \vecf \|_{ \mathbb{R}^n } }
	\cdot
	\frac { \Delta_j \vecf } { \| \Delta_j \vecf \|_{ \mathbb{R}^n } }
	\right.
	\\
	& \quad \quad
	\left.
	+ \,
	\| \Delta_j \vecf \|_{ \mathbb{R}^n }
	\left(
	\Delta_i^j \vecf \cdot \frac { \Delta_i \vecf } { \| \Delta_i \vecf \|_{ \mathbb{R}^n } }
	\right)
	-
	\| \Delta_i \vecf \|_{ \mathbb{R}^n }
	\| \Delta_j \vecf \|_{ \mathbb{R}^n }
	\right\}
	\\
	= & \
	\frac 1
	{
	\| \Delta_{ i+1 }^j \vecf \|_{ \mathbb{R}^n }
	\| \Delta_i^{ j+1 } \vecf \|_{ \mathbb{R}^n }
	}
	\left\{
	\left(
	\Delta_{ i+1 }^j \vecf \cdot \frac { \Delta_i \vecf } { \| \Delta_i \vecf \|_{ \mathbb{R}^n } }
	\right)
	\left(
	\Delta_i^j \vecf \cdot \frac { \Delta_j \vecf } { \| \Delta_j \vecf \|_{ \mathbb{R}^n } }
	\right)
	\right.
	\\
	& \quad \quad
	\left.
	+ \,
	\left(
	\Delta_i^j \vecf \cdot \frac { \Delta_i \vecf } { \| \Delta_i \vecf \|_{ \mathbb{R}^n } }
	\right)
	\left(
	\Delta_i^{ j+1} \vecf \cdot \frac { \Delta_j \vecf } { \| \Delta_j \vecf \|_{ \mathbb{R}^n } }
	\right)
	\right.
	\\
	& \quad \quad
	\left.
	- \,
	\| \Delta_i^j \vecf \|_{ \mathbb{R}^n }^2
	\frac { \Delta_i \vecf } { \| \Delta_i \vecf \|_{ \mathbb{R}^n } }
	\cdot
	\frac { \Delta_j \vecf } { \| \Delta_j \vecf \|_{ \mathbb{R}^n } }
	-
	\| \Delta_i \vecf \|_{ \mathbb{R}^n }
	\| \Delta_j \vecf \|_{ \mathbb{R}^n }
	\right\}
	.
\end{align*}
\par
The unit tangent vector $ \vect ( j+1, i , i+1 ) $ at $ \vecf ( \theta_{i+1} ) $ of the circle passing though
$ \vecf ( \theta_{ j+1 } ) $,
$ \vecf ( \theta_i ) $,
$ \vecf ( \theta_{ i+1 } ) $ is
\begin{align*}
	\vect ( j+1 , i , i+1 )
	= & \
	\frac
	{
	\| \Delta_i \vecf \|_{ \mathbb{R}^n }^2
	\Delta_{ i+1 }^{ j+1 } \vecf
	+
	\| \Delta_{ i+1 }^{ j+1 } \vecf \|_{ \mathbb{R}^n }^2
	\Delta_i \vecf
	}
	{
	\| \Delta_{ i+1 }^{ j+1 }\vecf \|_{ \mathbb{R}^n }
	\| \Delta_{ j+1 }^i \vecf \|_{ \mathbb{R}^n }
	\| \Delta_i \vecf \|_{ \mathbb{R}^n }
	}
	\\
	= & \
	\frac 1
	{
	\| \Delta_{ i+1 }^{ j+1 } \vecf \|_{ \mathbb{R}^n }
	\| \Delta_i^{ j+1 } \vecf \|_{ \mathbb{R}^n }
	}
	\left(
	\| \Delta_i \vecf \|_{ \mathbb{R}^n }^2
	\Delta_{ i+1 }^{ j+1 } \vecf
	+
	\| \Delta_{ i+1 }^{ j+1 } \vecf \|_{ \mathbb{R}^n }^2
	\frac { \Delta_i \vecf } { \| \Delta_i \vecf \|_{ \mathbb{R}^n } }
	\right)
	.
\end{align*}
The unit tangent vector $ \vect ( j, j+1, i+1 ) $ at $ \vecf ( \theta_{ i+1 } ) $ of the circle passing though $ \vecf ( \theta_j ) $,
$ \vecf ( \theta_{ j+1 } ) $,
$ \vecf ( \theta_{ i+1 } ) $ is
\begin{align*}
	&
	\vect ( j , j+1 , i+1 )
	=
	\frac
	{
	\| \Delta_{ j+1 }^{ i+1 } \vecf \|_{ \mathbb{R}^n }^2
	\Delta_{ i+1 }^j \vecf
	+
	\| \Delta_{ i+1 }^j \vecf \|_{ \mathbb{R}^n }^2
	\Delta_{ j+1 }^{ i+1 } \vecf
	}
	{
	\| \Delta_{ i+1 }^j \vecf \|_{ \mathbb{R}^n }
	\| \Delta_j \vecf \|_{ \mathbb{R}^n }
	\| \Delta_{ j+1 }^{ i+1 } \vecf \|_{ \mathbb{R}^n }
	}
	\\
	& \quad
	=
	\frac
	{
	\| \Delta_{ i+1 }^{ j+1 } \vecf \|_{ \mathbb{R}^n }^2
	\Delta_{ i+1 }^j \vecf
	-
	\| \Delta_{ i+1 }^j \vecf \|_{ \mathbb{R}^n }^2
	\Delta_{ i+1 }^{ j+1 } \vecf
	}
	{
	\| \Delta_{ i+1 }^j \vecf \|_{ \mathbb{R}^n }
	\| \Delta_j \vecf \|_{ \mathbb{R}^n }
	\| \Delta_{ i+1 }^{ j+1 } \vecf \|_{ \mathbb{R}^n }
	}
	\\
	& \quad
	=
	\frac
	{
	\left(
	\| \Delta_{ i+1 }^{ j+1 } \vecf \|_{ \mathbb{R}^n }^2
	-
	\| \Delta_{ i+1 }^j \vecf \|_{ \mathbb{R}^n }^2
	\right)
	\Delta_{ i+1 }^{ j+1 } \vecf
	-
	\| \Delta_{ i+1 }^{ j+1 } \vecf \|_{ \mathbb{R}^n }^2
	\Delta_j \vecf
	}
	{
	\| \Delta_{ i+1 }^j \vecf \|_{ \mathbb{R}^n }
	\| \Delta_j \vecf \|_{ \mathbb{R}^n }
	\| \Delta_{ i+1 }^{ j+1 } \vecf \|_{ \mathbb{R}^n }
	}
	\\
	& \quad
	=
	\frac
	{
	\left\{
	\Delta_j \vecf
	\cdot
	\left(
	2 \Delta_{ i+1 }^{ j+1 } \vecf
	-
	\Delta_j \vecf
	\right)
	\right\}
	\Delta_{ i+1 }^{ j+1 } \vecf
	-
	\| \Delta_{ i+1 }^{ j+1 } \vecf \|_{ \mathbb{R}^n }^2
	\Delta_j \vecf
	}
	{
	\| \Delta_{ i+1 }^j \vecf \|_{ \mathbb{R}^n }
	\| \Delta_j \vecf \|_{ \mathbb{R}^n }
	\| \Delta_{ i+1 }^{ j+1 } \vecf \|_{ \mathbb{R}^n }
	}
	\\
	& \quad
	=
	\frac
	1
	{
	\| \Delta_{ i+1 }^j \vecf \|_{ \mathbb{R}^n }
	\| \Delta_{ i+1 }^{ j+1 } \vecf \|_{ \mathbb{R}^n }
	}
	\\
	& \quad \qquad
	\times
	\left\{
	2
	\left(
	\Delta_{ i+1 }^{ j+1 } \vecf
	\cdot
	\frac { \Delta_j \vecf } { \| \Delta_j \vecf \|_{ \mathbb{R}^n } }
	\right)
	\Delta_{ i+1 }^{ j+1 } \vecf
	-
	\| \Delta_j \vecf \|_{ \mathbb{R}^n }
	\Delta_{ i+1 }^{ j+1 } \vecf
	-
	\| \Delta_{ i+1 }^{ j+1 } \vecf \|_{ \mathbb{R}^n }^2
	\frac { \Delta_j \vecf } { \| \Delta_j \vecf \|_{ \mathbb{R}^n } }
	\right\}
	.
\end{align*}
Consequently
\begin{align*}
	\cos \tilde \alpha_{ij}
	= & \
	\vect ( j+1 , i , i+1 ) \cdot \vect ( j , j+1 , i+1 )
	\\
	= & \
	\frac
	1
	{
	\| \Delta_i^{ j+1 } \vecf \|_{ \mathbb{R}^n }
	\| \Delta_{ i+1 }^j \vecf \|_{ \mathbb{R}^n }
	}
	\left\{
	2 \| \Delta_i \vecf \|_{ \mathbb{R}^n }
	\left(
	\Delta_{ i+1 }^{ j+1 } \vecf
	\cdot
	\frac { \Delta_j \vecf } { \| \Delta_j \vecf \|_{ \mathbb{R}^n } }
	\right)
	-
	\| \Delta_i \vecf \|_{ \mathbb{R}^n }
	\| \Delta_j \vecf \|_{ \mathbb{R}^n }
	\right.
	\\
	& \quad
	\left.
	- \,
	\| \Delta_i \vecf \|_{ \mathbb{R}^n }
	\left(
	\Delta_{ i+1 }^{ j+1 } \vecf
	\cdot
	\frac { \Delta_j \vecf } { \| \Delta_j \vecf \|_{ \mathbb{R}^n } }
	\right)
	+
	2
	\left(
	\Delta_{ i+1 }^{ j+1 } \vecf
	\cdot
	\frac { \Delta_i \vecf } { \| \Delta_i \vecf \|_{ \mathbb{R}^n } }
	\right)
	\left(
	\Delta_{ i+1 }^{ j+1 } \vecf
	\cdot
	\frac { \Delta_j \vecf } { \| \Delta_j \vecf \|_{ \mathbb{R}^n } }
	\right)
	\right.
	\\
	& \quad
	\left.
	- \,
	\| \Delta_j \vecf \|_{ \mathbb{R}^n }
	\left(
	\Delta_{ i+1 }^{ j+1 } \vecf
	\cdot
	\frac { \Delta_i \vecf } { \| \Delta_i \vecf \|_{ \mathbb{R}^n } }
	\right)
	-
	\| \Delta_{ i+1 }^{ j+1 } \vecf \|_{ \mathbb{R}^n }^2
	\frac { \Delta_i \vecf } { \| \Delta_i \vecf \|_{ \mathbb{R}^n } }
	\cdot
	\frac { \Delta_j \vecf } { \| \Delta_j \vecf \|_{ \mathbb{R}^n } }
	\right\}
	\\
	= & \
	\frac
	1
	{
	\| \Delta_i^{ j+1 } \vecf \|_{ \mathbb{R}^n }
	\| \Delta_{ i+1 }^j \vecf \|_{ \mathbb{R}^n }
	}
	\left\{
	2
	\left(
	\Delta_{ i+1 }^{ j+1 } \vecf
	\cdot
	\frac { \Delta_i \vecf } { \| \Delta_i \vecf \|_{ \mathbb{R}^n } }
	\right)
	\left(
	\Delta_{ i+1 }^{ j+1 } \vecf
	\cdot
	\frac { \Delta_j \vecf } { \| \Delta_j \vecf \|_{ \mathbb{R}^n } }
	\right)
	\right.
	\\
	& \quad
	\left.
	+ \,
	\| \Delta_i \vecf \|_{ \mathbb{R}^n }
	\left(
	\Delta_{ i+1 }^{ j+1 } \vecf
	\cdot
	\frac { \Delta_j \vecf } { \| \Delta_j \vecf \|_{ \mathbb{R}^n } }
	\right)
	-
	\| \Delta_j \vecf \|_{ \mathbb{R}^n }
	\left(
	\Delta_{ i+1 }^{ j+1 } \vecf
	\cdot
	\frac { \Delta_i \vecf } { \| \Delta_i \vecf \|_{ \mathbb{R}^n } }
	\right)
	\right.
	\\
	& \quad
	\left.
	- \,
	\| \Delta_{ i+1 }^{ j+1 } \vecf \|_{ \mathbb{R}^n }^2
	\frac { \Delta_i \vecf } { \| \Delta_i \vecf \|_{ \mathbb{R}^n } }
	\cdot
	\frac { \Delta_j \vecf } { \| \Delta_j \vecf \|_{ \mathbb{R}^n } }
	-
	\| \Delta_i \vecf \|_{ \mathbb{R}^n }
	\| \Delta_j \vecf \|_{ \mathbb{R}^n }
	\right\}
	\\
	= & \
	\frac
	1
	{
	\| \Delta_i^{ j+1 } \vecf \|_{ \mathbb{R}^n }
	\| \Delta_{ i+1 }^j \vecf \|_{ \mathbb{R}^n }
	}
	\left\{
	\left(
	\Delta_i^{ j+1 } \vecf
	\cdot
	\frac { \Delta_i \vecf } { \| \Delta_i \vecf \|_{ \mathbb{R}^n } }
	\right)
	\left(
	\Delta_{ i+1 }^{ j+1 } \vecf
	\cdot
	\frac { \Delta_j \vecf } { \| \Delta_j \vecf \|_{ \mathbb{R}^n } }
	\right)
	\right.
	\\
	& \quad
	\left.
	+ \,
	\left(
	\Delta_{ i+1 }^{ j+1 } \vecf
	\cdot
	\frac { \Delta_i \vecf } { \| \Delta_i \vecf \|_{ \mathbb{R}^n } }
	\right)
	\left(
	\Delta_{ i+1 }^j \vecf
	\cdot
	\frac { \Delta_j \vecf } { \| \Delta_j \vecf \|_{ \mathbb{R}^n } }
	\right)
	\right.
	\\
	& \quad
	\left.
	- \,
	\| \Delta_{ i+1 }^{ j+1 } \vecf \|_{ \mathbb{R}^n }^2
	\frac { \Delta_i \vecf } { \| \Delta_i \vecf \|_{ \mathbb{R}^n } }
	\cdot
	\frac { \Delta_j \vecf } { \| \Delta_j \vecf \|_{ \mathbb{R}^n } }
	-
	\| \Delta_i \vecf \|_{ \mathbb{R}^n }
	\| \Delta_j \vecf \|_{ \mathbb{R}^n }
	\right\}
	.
\end{align*}


\end{document}